\documentclass[11pt,reqno]{amsart}
\usepackage{amsfonts,amssymb,amsmath}
\usepackage{datetime}
\usepackage{xcolor}

\definecolor{fgreen}{RGB}{44,144, 14}

\renewenvironment{proof}{{\bfseries Proof.}}{\qed}

\numberwithin{equation}{section} 
\newtheorem{theorem}{Theorem}[section] 
 
\newtheorem{corollary}[theorem]{Corollary} 
\newtheorem{lemma}[theorem]{Lemma} 

\theoremstyle{definition}

\newtheorem{example}[theorem]{Example}

\usepackage[colorlinks=true, citecolor=cyan,urlcolor=blue,linkcolor=violet]{hyperref}
\def\R{\mathbb {R}}
\def\C{\mathbb {C}}
\def\N{\mathbb {N}}
\def\H{\mathbb {H}}

\def\d{\mathbf {d}}

\def\ab{\mathbf {a}}
\def\xb{\mathbf {x}}

\def\PC{\mathcal {P}}

\def\BC{\mathcal {B}}

\def\g{\mathfrak {g}}

\def\l{\mathfrak {l}}

\def\uf{\mathfrak {u}}

\def\lto{\longrightarrow}
\def\GL{\rm GL}

\newcommand{\secref}[1]{Section~\ref{#1}}
\newcommand{\thmref}[1]{Theorem~\ref{#1}}
\newcommand{\lemref}[1]{Lemma~\ref{#1}}

\newcommand{\corref}[1]{Corollary~\ref{#1}}

\def\U{{\rm U}_n^{\pm1}}
\begin{document} 
	
\title[Real adjoint orbits of the unipotent subgroup]{Real adjoint orbits of the unipotent subgroup}
\author[K. Gongopadhyay  \and C. Maity]{Krishnendu Gongopadhyay \and Chandan Maity
}
\address{Indian Institute of Science Education and Research (IISER) Mohali,
	Knowledge City,  Sector 81, S.A.S. Nagar 140306, Punjab, India}
\email{krishnendug@gmail.com, krishnendu@iisermohali.ac.in}
\address{Indian Institute of Science Education and Research (IISER) Mohali,
	Knowledge City,  Sector 81, S.A.S. Nagar 140306, Punjab, India}
\email{maity.chandan1@gmail.com, cmaity@iisermohali.ac.in}
\thanks{Maity was supported by NBHM post-doctoral fellowship during the course of this work.}
%\date{@\currenttime,  \today}

\begin{abstract}	
Let $G$ be a linear Lie group that acts on it's Lie algebra $\g$ by the adjoint action: ${\rm Ad}(g)X=gXg^{-1}$. An element $X\in \g$ is called {\it {\rm Ad}$_G$-real } if $-X ={\rm Ad}(g)X $ for some $g\in G$. An  {\rm Ad}$_G$-real  element $X$ is called {\it strongly {\rm Ad}$_G $-real}  if $-X = {\rm Ad}(\tau) X $ for some involution  $\tau\in G$.

 Let $K=\R$, $\C$ or $\H$. Let 
${\rm U}_n(K)$ be the group of unipotent upper-triangular matrices over $K$. Let $\mathfrak{u}_n (K)$ be the Lie algebra of ${\rm U}_n(K)$ that consists of  $n \times n$ upper triangular matrices with $0$ in all the diagonal entries. In this paper, we consider the Ad-reality of the Lie algebra $ \mathfrak u_n(K) $ that comes from the adjoint action of the Lie  group ${{\rm U}_n(K)}$ on $ \mathfrak u_n(K)$. We prove that there is no non-trivial {\rm{Ad}}$_{{\rm U}_n(K)}$-real element in $\mathfrak{u}_n (K)$. We also consider the adjoint action of the extended group ${\U}(K)$ that consists of all upper triangular matrices over $K$ having diagonal elements as $1$ or $-1$, and construct a large class of Ad$ _{\U( K)} $-real elements. As applications of these results, we recover related results concerning classical reality in these groups. 
\end{abstract}
\maketitle

\section{Introduction} 
Let $K=\R$, $\C$ or $\H$. Let ${\rm T}_n(K)$ be the subgroup of ${\rm GL}_n(K)$ consists of $n \times n$ upper triangular matrices over $K$, and 
${\rm U}_n(K)$ be the subgroup of ${\rm T}_n(K)$ consists of  matrices with $1$ in the diagonal entries. Let $\mathfrak{u}_n (K)$ be the Lie algebra of ${\rm U}_n(K)$ that consists of  $n \times n$ upper triangular matrices with $0$ in all the diagonal entries. 

 It is well-known in the literature that the number of conjugacy classes in ${\rm T}_n(K)$ is infinite for  $n \geq 6$, see \cite{DM}. Given an element $g$ in ${\rm T}_n(K)$, one may consider its Jordan canonical form ${\rm J}_g$ in ${\rm GL}_n(K)$. There are infinitely many conjugacy classes in ${\rm T}_n(K)$ corresponding to the same form ${\rm J}_g$. In general the problem of listing conjugacy class representatives in ${\rm T}_n(K)$ is considerably difficult. In the literature one often considers the action of ${\rm T}_n(K)$ on the Lie algebra $\mathfrak{u}_n(K)$ of upper triangular nilpotent matrices and tries to classify the conjugacy orbits. The Belitski\u{i}'s algorithm gives one such algorithm to classify such canonical forms. However, complete list of canonical forms using this algorithm is only known for lower values of $n$,  see  \cite{Ko}, \cite{CXLF}, \cite{TBH} for  $n\le 8$. 

It is a problem of related interest to classify the real conjugacy classes. Recall that an element $g$ in a group $G$ is called \emph{real} or \emph{reversible} if it is conjugate to $g^{-1}$ in $G$. The element $g$ is called \emph{strongly real} or \emph{strongly reversible} if $g$ is a product of two involutions in $G$. Equivalently if there is an involution $h\in G$ so that $g^{-1}=hgh^{-1}$, then $g$ is strongly real. Thus a strongly real element in $G$ is real, but a real element is not necessarily strongly real. 
 It has been a problem of wide interest to investigate the real and the strongly real elements in groups, see \cite{FS} for an exposition of this theme.

 For a Lie group $G$, one can define a notion of reality in the Lie algebra $\g$ of $G$.  Any Lie group $G$  acts on it's Lie algebra $\g$ via the $ {\rm Ad} $-representation. 
 An element $X\in \g$ is called {\it {\rm Ad}$_G$-real } if $-X ={\rm Ad}(g)X $ for some $g\in G$. An  {\rm Ad}$_G$-real  element $X$ is called {\it strongly {\rm Ad}$_G $-real}  if $-X = {\rm Ad}(\tau) X $ for some involution  $\tau\in G$, see \cite[Definition 1.1]{GM}.  Observe that if $X\in \g$ is  {\rm Ad}$_G$-real, then $\exp X$ is real in $G$. 

For a  linear  Lie group $ G $, the $ {\rm Ad} $-action is given by $ {\rm Ad}(g)X=gXg^{-1}$.  The the notion of Ad-reality comes from the conjugation action of $G$ on $\g$. 
 We have  applied the notion of Ad-reality in \cite{GM} to classify real and strongly real unipotent elements in a classical simple Lie group.

In this exposition  we consider the Ad-reality of the Lie algebra $ \mathfrak u_n(K) $ that comes from the action of the linear Lie  group ${{\rm U}_n(K)}$ on $ \mathfrak u_n(K)$. We prove the following result. 
\begin{theorem}\label{1.1} There is no non-trivial {\rm{Ad}}$_{{\rm U}_n(K)}$-real element in $\mathfrak{u}_n (K)$.
\end{theorem}

The exponential map is a diffeomorphism for the unipotent subgroup ${\rm U}_n(K)$. So, the ${\rm Ad}_{{\rm U}_n(K)}$-reality in $ \mathfrak{u}(n,K) $ and the classical reality in ${\rm U}_n(K)$ are equivalent under the exponential map. From this the following result follows. 
\begin{corollary}\label{cor1.2}
 The only real element in the unipotent group ${\rm U}_n(K)$ is the identity \hbox{element}. 
\end{corollary}

When $K=\R$ or $\C$,  the above corollary was proven in \cite[Theorem 1.3]{S} using a different method that involves  matrix computations. Our method is completely different because of the use of the Lie theoretic ideas. Our proof is also much shorter. Another advantage of our method is that the proof is also valid when $K$ is the division ring of the real quaternions, and hence it provides a uniform approach  irrespective of the underlying skew-field. The method in \cite{S} is unlikely to hold for matrices over the quaternions due to the non-commutativity of the quaternions. 

Since any unipotent Lie group $ {\rm N} $ can be embedded in $ {\rm  U}_n(K) $ for some $n$, as an immediate consequence  we have the following result. 
\begin{corollary}
There is no non-trivial real element in a unipotent Lie group ${\rm N}$. 
\end{corollary}
We note that the above result has also been proved in \cite[Lemma 4.6]{C} using a more general machinery from the theory of linear algebraic groups. 

\medskip Even though we see that there is no real unipotent element in ${\rm U}_n(K)$, it may happen that the unipotent element becomes real in an enlarged group containing ${\rm U}_n(K)$.  Let ${\U}(K)$ be the group of all upper triangular matrices over $K$ having diagonal elements as $1$ or $-1$. 
Note that if an upper triangular  matrix $g$ in $  {\rm T}_n(K) $ is strongly real, then $ g $ has to be in $ {\U}(K) $. This motivates us  to consider the conjugation action of ${\U}(K)$ on the Lie algebra $\mathfrak{u}_n( K)$. 
 
Let $ \uf_n^\star(K) $ be the subset of $\uf_n(K)$ consisting of the elements with non-zero entries at $ (i, i+1)^{\rm th}$ place for $ i=1, \dots , n-1 $  (  entries  above the main diagonals are non-zero). We prove the following. 

\begin{theorem}\label{thm1.4}
Every element in $ \uf_n^\star(K) $ is {\rm Ad}$ _{\U(K)} $-real.
\end{theorem}
  In \cite{S} some class of real unipotent elements in ${\U}(F)$ has been constructed for $F=\R$ or $\C$. Applying the above theorem we get the following corollary that gives certain types of real elements in $ \U(K)$. This extends the class of real elements in \cite{S}  over the quaternions. 

  Let ${\rm U}_n^\star(K) $ be the set of elements in ${\rm U}_n(K)$ consisting of  
  elements with non-zero entries at $ (i, i+1)^{\rm th}$ place for $ i=1, \dots , n-1 $  (entries  above the main diagonals are non-zero).  
  \begin{corollary} \label{thm-U}
  	Every element in  ${\rm U}_n^\star(K) \sqcup\big(  -{\rm U}_n^\star(K) \big)$ is real in $ \U(K) $. 
  \end{corollary}
  
    For lower values of $n$, using  Belitski\u{i}'s algorithm canonical forms for  ${\rm Ad}_{{\rm T}_n(\C)}$-orbits of nilpotent  elements in $\uf_n(\C)$ has been obtained in the literature, e.g. \cite{CXLF}, \cite{TBH}, \cite{Ko}.  We have checked that each of the representatives are, in fact, strongly real over $K$. 
   All the examples of conjugacy class representatives  of unipotent elements in ${\U}(K)$ that we encountered are strongly real. We could not construct any non-real  unipotent  element in ${\U}(K)$.

The paper is organized as follows. In \secref{sec-notation} we fix the notation and recall some background. \thmref{1.1} is proved in \secref{sec-3}. In \secref{sec-4} we have proved \thmref{thm1.4} and \corref{thm-U}. 
Finally we write down some examples of real unipotent elements in ${\U}(K)$  for $n\leq 8$ in \secref{sec-5}. 

\section{Notation and Background}\label{sec-notation}
In this section we will fix some notation.
The Lie groups will be denoted by the capital letters,
  while the Lie algebra of a Lie group will be denoted
   by the corresponding lower case German letter.
Let $V$ be a right vector space of dimension $n$ over $K$,
 where $K$ is, as before, $\R$, $ \C $ or $ \H$. 
Let ${\g\l} (V)$ be the real algebra of {\it right $K$-linear maps}
 from $V$ onto $V$. After choosing a basis for $V$,  the elements
  of $\g\l (V)$ can be  represented by matrices over $K$. 
For a $K$-linear map $A \,\in\, \g\l(V)$ and an ordered $K$-basis $\BC$ of $V$,
the {\it matrix of $A$ with respect to $\BC$} is denoted by $[A]_{\BC}$.
Let  ${\rm GL(V)}$ denote the {\it group of invertible} right  $ K $-linear maps from ${\g\l} (V)$.

For two disjoint ordered sets
$(v_1,\, \ldots ,\,v_n)$ and $(w_1,\, \ldots ,\,w_m)$, we denote
the ordered set $(v_1,\, \ldots ,\,v_n,\, w_1,\, \ldots ,\, w_m)$  by
$$(v_1, \,\ldots ,\,v_n) \vee (w_1,\, \ldots ,\,w_m)\, .$$
A {\it partition} of a positive integer $n$ is an object of the form
$\d:=[ d_1^{t_{1}},\, \ldots ,\, d_s^{t_{s}} ]$, where $t_i,\, d_i \,\in\, \N$
 such that $\sum_{i=1}^{s} t_{i} d_i \,=\, n$, and $ d_1>\cdots>d_s>0$.
 Let $\PC(n)$ denote the {\it set of all partitions of $n$}. 

Now we include without proof a basic result.
\begin{lemma}\label{lem-eigen-value}
Let $A\in {\GL}_n(K)$ be a block matrix  of the form $A \,=\, (A_{ij})$, where $A_{ij}$'s are block matrices with $A_{ij}=0$ for $i>j$. Then the eigen values of $A$ are precisely union of the eigen values of the diagonal blocks $A_{ii}$ $(i=1,\ldots, m)$ of $A$.
\end{lemma}

\section{Proof of \thmref{1.1}}\label{sec-3}
Here we first prove a basic result which will be crucially used in the proof of \thmref{1.1}. Before that let us consider an example. 
\begin{example}
Let $ J\,:=\, \begin{pmatrix}
	0 & 1 &0\\
	  & 0 &1\\
	  &  & 0\\
	  &  &  & 0&1\\
	  & &   & &  0\\
	  &  &  &  &  &0
\end{pmatrix} $ be a nilpotent $ 6\times 6 $ Jordan matrix. The underlying vector space $ K^6 $ has a basis of the form : $ \{J^2e_3, \,Je_3,\, e_3,\, Je_5,\, e_5,\,e_6 \} $. We now consider the following ordered basis of $ K^6 $: 
\begin{align}\label{ordered-basis-B-exmp}
\BC:=\, \{J^2e_3, \,Je_5,\, e_6,\, Je_3,\, e_5,\,e_3 \}\,=\, (v_1,\dots, v_6)\,,{\rm say}.
\end{align}
If  $g\in {\GL}(K^n)$ be so that  $-J=g J g^{-1}$, then $ g$ is completely determined by it's action on $ e_i $ for $ i=3,5,6 $. Write $ g e_r\,=\, \sum_{i,\,l} c^l_{i\,r} J^le_i$ for $ r=3,5,6 $, where $ c^l_{i\,l} \in K$. By multiplying $J$ in both the sides suitably and using the fact that $ \BC $ is an basis, it follows that 
\begin{itemize}
	\item $~ g e_6\,=\,  c^2_{3\,6} J^2e_3\,+\,  c^1_{5\,6} Je_5\,+\, c^0_{6\,6}e_6$.
	\item $~ g e_5\,=\,  c^2_{3\,5} J^2e_3\,+\,  c^1_{5\,5} Je_5\,+\, c^0_{6\,5}e_6\,+\,  c^1_{3\,5} J^1 e_3\,+\, c^0_{5\,5} e_5 $.
	\item $~ g e_3\,=\,  c^2_{3\,3} J^2e_3\,+\,  c^1_{5\,3} Je_5\,+\, c^0_{6\,3}e_6\,+\,  c^1_{3\,3} J^1 e_3\,+\, c^0_{5\,3} e_5 \,+\, c^0_{3\,3} e_3$.
\end{itemize}
Therefore, 
$$g\big( \text{Span}_K\{v_1,\dots, v_i\} \big) \, \subseteq  \text{ Span}_K\{v_1,\dots, v_i\}\,, \quad {\rm for\, } \  i=1, \dots 6 \,,$$
where $ (v_1,\dots,v_6) $ as in \eqref{ordered-basis-B-exmp}. 
 In particular, the matrix $ [g]_\BC $ is of the form :
$$
[g]_\BC\,=\,
\begin{pmatrix}
c^0_{3\,3} &-c^1_{3\,5}  & c^2_{3\,6} & -c^1_{3\,3} & c^2_{3\,5}& c^2_{3\,3}\\
  	       &-c^0_{5\,5} & c^1_{5\,6}& -c^0_{5\,3}  & c^1_{5\,5} & c^1_{5\,3} \\
  	       &   &c^0_{6\,6} & 0 & c^0_{6\,5}  &  c^0_{6\,3} \\  
  	       &   && -c^0_{3\,3}  & c^1_{3\,5}  &  c^1_{3\,3} \\  
  	       & &  &   & c^0_{5\,5}  &  c^0_{5\,3} \\
  	       & &  &   &   &  c^0_{3\,3} \,.
\end{pmatrix}
$$
Note that $[g]_\BC$ can not be a unipotent matrix as both the non-zero scalars $ c^0_{3\,3}$ and $ -c^0_{3\,3}$ are eigen values of $ g $.
\end{example}

Let $J$ be a nilpotent  Jordan matrix and $-J=g J g^{-1}$ for some invertible matrix $g$. The next lemma generalizes the above example and 
gives a necessary condition about such conjugating element $g$. 

\begin{lemma}\label{lem-conjugating-elt-nonunipotent}
 Let $J$ be a nilpotent Jordan block matrix in ${\rm M}_n(K)$, and $g\in  {\GL}_n(K )$ so that $-J=g J g^{-1}$. Then there is an eigen value $\lambda$ of $g$ $(\lambda \in  \bar K)$ so that $-\lambda$ is also an eigen value of $g$. In particular, $g$ can not be a unipotent matrix. 
\end{lemma}

\begin{proof}
Recall that given a $ n\times n $ nilpotent Jordan block matrix $J$, one can associate a  partition $\d=[d_1^{t_1}, \ldots, d_s^{t_{s}}]$ of $ n $. 
 For the nilpotent matrix $J$, there exists column vectors $e_i^{d_r}\in K^n$ so that the underlying vector space $K^n$ has a basis of the form $\{J^le_i^{d_r}\, \mid \,  1\le r \le s, 1\le i \le t_r,  0\le l \le d_r-1\}$. Next we will define a suitable ordering for this basis. Let 
$$\BC^l(d_r) \,:=\, (J^l e^{d_r}_1,\, \ldots ,\,J^l e^{d_r}_{t_r})\quad  {\rm for } \quad 0\,\leq\, l \,\leq\, d_r-1, 1\le r\le s\,.$$ 
Define 
\begin{equation*}\label{old-ordered-basis}
	\BC(j) \,:=\, \BC^{d_1-j} (d_1) \vee    \cdots \vee  \BC^{d_s-j} (d_s)  ,\ \text{ and }\  \BC\,:=\, \BC(1) \vee \cdots \vee  \BC(d_1)\, .
\end{equation*}
Then $\BC$ is an ordered basis of  $ K^n $. Using the relation $(-J)^r=g J^r g^{-1}$, it follows that 
\begin{itemize}
	\item  $g\big({\rm Span}_K\{\BC^{d_1-1}(d_1)\vee \cdots\vee \BC^{d_r-1}(d_r) \}\big) \,\subseteq \, {\rm Span}_K\{\BC^{d_1-1}(d_1)\vee \cdots\vee \BC^{d_r-1}(d_r) \} $\\  for $ r=1,\ldots, s $.
	\item $g\big( \text{Span}_K\{\BC(1)\vee \cdots \vee \BC(t) \}\big) \,\subseteq \, \text{Span}_K\{\BC(1)\vee \cdots \vee \BC(t) \}$ for $t=1,\ldots , d_1$. 
\end{itemize}
In particular, 
$$
g\big(\text{Span}_K\{\BC^{d_1-1}(d_1)\} \big)\, \subseteq \, \text{Span}_K\{\BC^{d_1-1}(d_1)\}.
$$ 
The matrix  $[g]_\BC$ is a block upper triangular matrix with $ (d_1+\cdots +d_s) $-many diagonal blocks.   
 Write $ [g]_\BC = (g_{ij}) $, where $g_{ij}$ are  block matrices, and the first block  $g_{11}$ is given by $ [g|_{\text{Span}\{\BC^{d_1-1}(d_1)\}}]_{\BC^{d_1-1}(d_1)}$.
Note that the order of  the first $s$-many diagonal blocks $g_{jj}$  of $[g]_\BC$ are $ t_{j}\times t_{j} $,  for $j=1,\ldots , s$. 
The matrix  $[g]_\BC$ is of the following form:
$$
[g]_\BC\,=\,
\begin{pmatrix}
 g_{1\,1} &  \cdots &  \ \  \cdots  &g_{1\,n}\\
 &\ddots &   &\vdots \\
 &&   g_{r\,r} \  \cdots & g_{r\,n}\\     
 & &\qquad  \ddots   &\vdots \\
  & && g_{n\,n}
       \end{pmatrix}\,.
$$  
From the definition of the basis $ \BC $, it follows   that the $(s+1)^{\rm th}$  diagonal block matrix  $g_{s+1\,\,s+1}$   is  $-g_{11}$.
%, and $g_{n \,n}= (-1)^{d_1-1}g_{1\,1}$.
 Thus if $\lambda\in \bar K$ is an eigen value of $g_{1\,1}$, then $-\lambda$ is  an eigen value of $g_{s+1\,\,s+1}$. The proof  now follows using  Lemma \ref{lem-eigen-value}.
\end{proof}

{\bf Proof of \thmref{1.1}.~} 
Let $0\neq X\in \mathfrak{u}_n (K)$ be a Ad$ _{{\rm U}_n(K)} $-real element. Then $-X= \alpha X \alpha^{-1}$ for some $\alpha \in {\rm U}_n(K)$. Since $X$ is a nilpotent matrix, there is a $\beta\in {\GL}_n(K)$ so that  $J_X = \beta X \beta^{-1}$, where $J_X$ is the Jordan Canonical form of $X$. Hence $-J_X = (\beta \alpha \beta^{-1}) J_X  (\beta \alpha^{-1} \beta^{-1})$.  
This contradicts  \lemref{lem-conjugating-elt-nonunipotent}, as all the eigen values of  $\beta \alpha \beta^{-1}\in {\GL}_n(K)$ are $1$.  This completes the proof. \qed

\section{Proof of Theorem \ref{thm1.4}}\label{sec-4}
In this section we will consider  Ad-action of $ {\U}(K) $ on $ \mathfrak{u}_n (K) $. 
Recall that $ \uf_n^\star(K) $ be the subset of $\uf_n(K)$ consisting of the elements with non-zero entries at $ (i, i+1)^{\rm th}$ place for $ i=1, \dots , n-1 $  (  entries  above the main diagonals are non-zero).

{\bf Proof of \thmref{thm1.4}.} 
	We will use induction on $ n $,  order of  the matrix.  For $n=2,3 $ every element $ \uf_n(K) $ is ${\rm Ad}_{\U( K)} $-real, see Example \ref{exmp-2}.
	Now we will assume that the statement is true for the matrix in $ \uf_{n-1}^\star(K) $. 
	Let $ X\,=\, \begin{pmatrix}
		X_{n-1}& \xb\\
		&0
	\end{pmatrix} \in \mathfrak{u}_n^\star(K)$, where $ X_{n-1}\in\uf_{n-1}^\star(K)$, and $ \xb \in K^{n-1} $.
	By induction hypothesis, 
	 $g_{n-1}X_{n-1}g_{n-1}^{-1} \,=\, -X_{n-1}$ for some 
	 $g_{n-1}\in {\rm U}^{\pm1}_{n-1}(K)$. 
	
	Let $ g\,:=\, \begin{pmatrix}
		g_{n-1}& \ab\\
		&\epsilon
	\end{pmatrix}\in {\U}(K)$, where $ \ab\in K^{n-1}$ and $\epsilon\in \{ 1, -1\}$. 
	We want to choose $ \ab$ and $\epsilon $ so that  
	$$ g Xg^{-1}=-X .$$
	The relation $ gXg^{-1}=-X $ is equivalent to the the following two equations:
	\begin{align}
		g_{n-1}X_{n-1} +X_{n-1}g_{n-1}\,&=\,0\, \label{eq-1}\\
	  X_{n-1} \ab \,&\, =\,  -( g_{n-1}+\epsilon\, {\rm Id})\xb\,. \label{eq-2}
	\end{align}
Note that \eqref{eq-1} is already true.
 Thus we want to choose $ \ab$ and  $\epsilon$ so that \eqref{eq-2} hold.
 Write $ X_{n-1}=(x_{ij}) $ and $ \ab=(a_1, \dots, a_{n-1})^T $. 
 Set $ \epsilon = $ negative  of $( n-1, n-1)^{\rm th}$ entry of $ g_{n-1} $.
 Then \eqref{eq-2} is equivalent to the following system of equations.
\begin{equation}\label{system-eqn}
\sum_{\substack{1\le i<j\le n-1\\ i\leq n-2}} x_{ij}a_{j} \,=\, b_i\,,
\end{equation}
where $ (b_1,\dots, b_{n-1})^T:= -( g_{n-1}+\epsilon\, {\rm Id})\xb$.
Since $ x_{i\ i+1} \neq 0$ for $ i=1,\dots, n-2 $, the above system of equations \eqref{system-eqn} has a solution.
 Hence, $X$ is Ad$ _{\U(K)} $-real.
\qed

 The next lemma is basic result. We include a proof for the sake of completeness.
\begin{lemma}\label{exp}
	The exponential map $ \exp $ maps $  \uf_n^\star(K)$ bijectively onto  $ {\rm U}_n^\star(K) $.
\end{lemma}
\begin{proof}
The injectivity follows from the face that the exponential map 
$$ \exp\colon \uf_n(K)\,\lto \, {\rm U}_n(K) $$
is injective. For surjectivity, let $ u\in {\rm U}_n^\star(K) $. Then $ ({\rm Id} -u)^{n} =0$. Define
$$ X:= -({\rm Id} -u) - \frac{({\rm Id} -u)^2}{2} - \dots - \frac{({\rm Id} -u)^{n-1}}{n-1}\, .$$
Then $ X\in  \uf_n^\star(K)  $ and $ \exp (X) =u$.
\end{proof}

{\bf Proof of \corref{thm-U}.}
Enough to prove for any $ g\in {\rm U}_n^\star(K)$. Using \lemref{exp},  $ g=\exp X $ for some $ X\in \uf_n^\star(K)$. In view of  \thmref{thm1.4}, $ -X=hXh^{-1}$ for some $ h\in \U(K) $. Thus $ g^{-1}=hgh^{-1} $. This completes the proof.  \qed

\section{Examples of Strongly real Unipotent Elements in ${\U}(K)$ for lower $ n $ }\label{sec-5}
 
In \cite{CXLF},   canonical forms of upper triangular nilpotent matrices in  $\uf_n(K) $ under ${\rm Ad}_{{\rm T}_n(\C)}$-action were given for $ n\le 8 $.
In this section, we will verify that every such nilpotent matrix in $\uf_n(K) $ is strongly $ {\rm Ad }_{\U(K)} $-real. 
  Using the exponential map it follows that the corresponding unipotent element in $ \U(K) $ is strongly real. 

\begin{example}\label{exmp-2}
	$ \mathbf{(1).}$	For $n=2$, any   element is of the form  $X=\begin{pmatrix}
		0&x\\
		&0\end{pmatrix}$. 
	Let $g=\begin{pmatrix}
		1\\&-1
	\end{pmatrix}\in {\rm U^{\pm1}_2}(K)$. 
	Then $gXg^{-1}=-X$ and $g^{2}= {\rm Id}$.  Thus $ X $ is strongly Ad-real. 

$ \mathbf{(2).}$ For $n=3$, let $X= \begin{pmatrix}
                        0&a&b\\
                        &0&c\\
                        &&0
                       \end{pmatrix}\in \mathfrak{u}_3(K)
$ be any arbitrary element. 
Now we consider two cases. If either $a$ or $b$ non-zero. Then set $g := \begin{pmatrix}
                        1&p& -pr/2   \\
                        &-1&r\\
                        &&1
                       \end{pmatrix}\in  {\rm U^{\pm1}_3}(K)$ where   $p,r$ satisfy $cp+ar+2b=0$. 
 Finally if  $a=0=b$, then take $g:={\rm diag}(1,1,-1 )$. Thus  every nilpotent element is strongly Ad$_{ {\rm U^{\pm1}_3}(K)}$-real. \qed
 \end{example}

The Ad(${\rm T}_n(K)$)-orbits of any element in $\mathfrak{u}_n (K)$ are classified in \cite[Theorems 2.1 -- 2.4]{CXLF} for $n\leq 6$, and partially for $n=7,8$; see also \cite[Theorem 2]{Ko} for $n=5$.
We have verified that each of the nilpotent matrix in \cite{CXLF} is strongly Ad-real, and  list down an involution $ \sigma $ for each of such nilpotent matrix $ X $ in the same order as given in \cite{CXLF} so that $\sigma X \sigma^{-1}=-X$. 

{\it Involutions for the elements in \cite[Theorem 2.1]{CXLF}:-} 
\begin{align*}
& {\rm diag }(1,-1),\, {\rm diag }(-1,1,-1),\, {\rm diag }(1, -1,1,-1) ,\, {\rm diag }(1, -1,-1,1)     ,\, {\rm diag }(1, -1, 1,-1,1),\\
 &    {\rm diag }(1, -1, -1,1,1),\, {\rm diag } (1, -1, -1,1,-1)     ,\, {\rm diag }(1, -1, 1,1,-1)   ,\, {\rm diag }(1, 1, -1,-1,1)
\end{align*}

{\it Involutions for the elements in \cite[Theorem 2.2]{CXLF}:-} 
\begin{align*}
&{\rm diag }(1, -1, 1,-1,1,-1),\,  {\rm diag } (1, -1, -1,1,-1,1),\,  {\rm diag }(1, -1, 1,-1,1,-1),\\
&{\rm diag }(1, 1, -1,-1,1,-1),\,  {\rm diag }(1, -1, 1,1,-1,1),\, \ \  {\rm diag } (-1, 1, 1,-1,-1,1) ,\\
&{\rm diag }(-1, -1, 1,1,1,-1),\,  {\rm diag } (-1, 1, 1,-1,-1,1),\,  {\rm diag } (-1, 1, 1,-1,-1,1) ,\\
&{\rm diag }(1, -1,- 1,1,1,-1),\,  {\rm diag } (-1, 1,-1,1,1,-1),\,   {\rm diag } (-1, 1, -1,1,1,-1),\\
&{\rm diag }(-1, 1, 1,1,-1,-1),\,  {\rm diag } (1, 1,-1,-1,1,1),\, \ \  {\rm diag } (1, -1, 1,1,-1,-1) ,\\
&{\rm diag }(-1, 1, 1,-1,1,-1),\,  {\rm diag } (-1, 1, 1,-1,-1,1),\,  {\rm diag } (-1, 1, 1,-1,-1,1)\,.
\end{align*}

{\it Involutions for the elements in \cite[Theorem 2.3]{CXLF}:-} 
\begin{align*}
&{\rm diag }(1, 1, -1,-1,-1,1,1),\,  {\rm diag } (-1, 1, 1,-1,-1,1,1),\,  {\rm diag } (1, -1, -1,1,1,-1,1),\\
&{\rm diag }(1, -1, -1,1,1,1,-1),\,  {\rm diag } (-1, 1, 1,1,-1,-1,1),\,  {\rm diag } (1, -1, 1,1,-1,-1,1),\\
&\,{\rm diag }(1, 1, -1,-1,1,1,-1),\,  
{\scriptsize  { 
\begin{pmatrix}
                                    1&&-2\lambda &&&&\\
                                    &1\\
                                    &&-1\\
                                    &&&1\\
                                    &&&&-1\\
                                    &&&&&-1\\
                                    &&&&&&1\\
                                   \end{pmatrix}\, .
                             }   }
                             \end{align*}
The nilpotent  $8\times 8$ matrix involving two parameters $\mu$ and $\lambda$ given in \cite[p. 143]{CXLF} is strongly Ad-real. The corresponding involution is   diag$(1,-1,-1,1,1,-1,-1,1 )$.

\medskip 

The above examples shows that for lower values of $ n $  almost all nilpotent matrices in $\mathfrak{u}_n( K)$ are strongly Ad$_{ \U(K)}$-real. However, we did not found any non-real unipotent element in the group $ \U(K) $.   Generalizing the method in Example 5.2, one can start with a nilpotent matrix for lower values of  $n$ and see this by hand computation. However, with large $n$, the complexity grows and there are difficulties in generalizing these examples.

\end{document}